\documentclass[reqno, 12pt]{amsart}
\usepackage{graphicx,overpic,epsf,amssymb,amscd,amsfonts}

\newtheorem{thm}{Theorem}[section]

\newtheorem{prop}[thm]{Proposition}
\newtheorem{defn}[thm]{Definition}

\newtheorem{rmk}[thm]{Remark}

\newcommand{\C}{\mathbb{C}}
\newcommand{\R}{\mathbb{R}}
\newcommand{\Z}{\mathbb{Z}}
\newcommand{\Q}{\mathbb{Q}}

\newcommand{\M}{\mathcal{M}}

\newcommand{\del}[1]{{\partial_{#1}}}

\newcommand{\ev}{{\rm ev}}
\newcommand{\vir}{{\rm vir}}

\newcommand{\be}{\begin{enumerate}}
\newcommand{\ee}{\end{enumerate}}
\begin{document}
%%%%%%%%%%%%%%%%%%%%%%%%%%%%%%%%%%%%%%%%

\title{Gravitational Descendants and Linearized Contact Homology}

\author{Jian He}
\address{Universit\'{e} Libre de Bruxelles, Bruxelles, Belgium}
\email{jianhe@ulb.ac.be}
\date{October, 2012.}

\begin{abstract}
In this paper we prove a recursion relation between the the one-point genus-$0$
gravitational descendants of a Stein domain $(M,\partial M)$. This relation is best 
described by the degree $-2$ map $D$ in the linearized 
contact homology of $\partial M$, arising from the Bourgeois--Oancea 
exact sequence between symplectic homology of $M$ and linearized contact 
homology of $\partial M$. 
All one-point genus-$0$
gravitational descendants can be reduce to the one-point 
Gromov--Witten invariants via iterates of $D$.

%This gives a new geometrical interpretation
%of the Bourgeios--Oancea map $D$.
\end{abstract}

\maketitle

\section{Introduction} \label{section: introduction}

Let $(M^{2n},\omega)$ be a closed symplectic manifold
with a compatible almost complex structure $J$. Define
$\mathcal{M}_{g,n}(M,\beta)$ to be the moduli space 
of stable $J$-holomorphic maps $u\colon S \to M$ in the homology
class $\beta\in H_2(M)$. 
More precisely, $\mathcal{M}_{g,n}(M,\beta)$ consists of maps 
$u\colon S\to M$ such that 
$$J\circ du = du \circ i$$ 
where
$S := (S; p_1,\dots, p_n; i)$ is a Riemann surface of genus $g$, with
$n$ marked points $\{p_1,\dots, p_n\}$, and complex structure $i$. 
Let $\overline{\mathcal{M}}_{g,n}(M,\beta)$ be the compactification 
in the sense of Gromov.

At each element $u\colon (S; p_1,\dots, p_n;i)\to M$ of $\overline{\mathcal{M}}_{g,n}(M,\beta)$, the cotangent 
space to $S$ at the point $p_i$ is
a complex line. These cotangent 
spaces patch together to form a complex line 
bundle $L_i$ over $\overline{\mathcal{M}}_{g,n}(M,\beta)$,
called the $i$-th tautological line bundle. 
Denote its first Chern class by $\psi_i = c_1(L_i)$.

Given classes $\{\theta_i\}_{i=1}^n$ in $H^*(M)$,
the gravitational descendants of $M$ are defined by
\begin{equation}
\langle\tau_{a_1}(\theta_1)\dots\tau_{a_n}(\theta_n)\rangle_{g,\beta}^M 
:=\int_{[\overline{\mathcal{M}}_{g,n}(M,\beta)]^{\vir}} \ev_1^*(\theta_1)\cup \psi_1^{a_1}\cup\cdots\cup\ev_n^*(\theta_n)\cup \psi_n^{a_n}.
\end{equation}

Gromov--Witten invariants, or correlators, are those where 
$a_1=\dots =a_n=0$.
Descendants are symplectic invariants of $M$, their values are 
independent of the compatible complex structure $J$.

In the framework of Symplectic Field Theory (SFT) of Eliashberg, Givental and Hofer 
\cite{EliashbergGiventalHofer}, the theory of $J$-holomorphic curves, and hence
gravitational descendants, can be generalized to
symplectic manifolds $M$ with contact type boundary $\partial M$. 
The full SFT contains information about all the moduli spaces, and
is therefore difficult to access. However invariants can still be extracted 
by considering various simplifications of the full SFT model.

One such invariant is the linearized 
contact homology of $\partial M$ with respect to the filling $M$, $HC(\partial M)$.
In \cite{Oancea}, Bourgeois and Oancea established a Gysin-type exact
sequence relating the symplectic homology of $M$, $SH(M)$, 
and the linearized contact homology of $\partial M$. Furthermore, the degree $-2$
map $D$ can be described purely by counts of holomorphic curves in $M$, 
without reference to symplectic homology.

\begin{thm}
There exists a long exact sequence  
\begin{equation} \label{eq:intro} 
\ldots \rightarrow SH_{k-(n-3)}^+(M) \rightarrow HC_{k}(\partial M) 
\stackrel D \rightarrow HC_{k-2}(\partial M) \rightarrow SH_{k-1-(n-3)}^+(M) 
\rightarrow \ldots 
\end{equation} 
\end{thm}

\begin{rmk} \emph{Since we do not assume $c_1(M)$ to vanish, it is necessary to define $HC(\partial M)$ 
over the Novikov ring $\Lambda_{\omega}$ with $\Q$-coefficients, consisting
of formal linear combinations 
$$\lambda:=\sum_{A\in H_2(M;\Z)} \lambda_A e^A,\ \ \lambda_A\in\Q$$ 
such that 
$$ \#\{ A | \lambda_A\neq 0, \omega(A) \le c\} <\infty, \ \ \forall c>0 $$ 
Multiplication in $\Lambda_{\omega}$ is given by formal 
power series multiplication.}
\end{rmk}

In this paper we will focus on the genus-$0$ one-point invariants of 
$(M,\partial M)$, from now on $g=0$ and $n=1$ will be omitted from subscripts.
As we will see in section 2, each gravitational descendant can be interpreted 
as a homomorphism
\begin{equation}
\langle\tau_{a}(\theta)\rangle \colon HC(\partial M) \longrightarrow \Lambda_{\omega}
\end{equation}

The main result of this paper is the following identity:
\begin{thm}\label{main}
Let $M^{2n}, n\geq 3$ be a subcritical Stein domain, $c \in HC(\partial M)$,
$\theta\in H^*(M,\partial M)$, and $D$ the degree $-2$ map of 
Bourgeois--Oancea, then
\begin{equation}\label{maineq}
\langle\tau_{l}(\theta)\rangle(c) =\frac{1}{l} \langle\tau_{l-1}(\theta)\rangle(Dc)
\end{equation}
\end{thm}

\begin{rmk}\emph{This fits into the works of Bourgeois and Oancea as follows. 
Linearized contact homology is isomorphic to the $S^1$-equivariant symplectic homology of $M$, $SH_k^{+,S^1}(M)$, 
defined in \cite{S1symplectic}. On the symplectic homology side, the exact sequence \eqref{eq:intro}
becomes very much like the Gysin sequence for ordinary and $S^1$-equivariant homology for 
a space $X$ with an $S^1$-action, where the degree $-2$ map is given in spirit by ``capping with an
Euler class'':
\begin{equation} \label{eq:SGysin} 
 \ldots \to SH_k^+(M) \to SH_k^{+,S^1}(M) \stackrel {\cap e} \to 
SH_{k-2}^{+,S^1}(M) \to SH_{k-1}^+(M) \to \ldots 
\end{equation} 
We also have the tautological long exact sequence 
\begin{equation}  
 \ldots \to SH_k^{-,S^1}(M) \to SH_k^{S^1}(M) \to SH_{k}^{+,S^1}(M) \stackrel{\partial} \to SH_{k-1}^{-,S^1}(M) \to \ldots 
\end{equation} 
Since the $S^1$-action on $SH^-$ is trivial,
$SH_*^{-,S^1}(M) = H(M)_*[t]$ is the Morse homology of $M$ tensored with the homology of $\C P^{\infty}$, where $t$ is the 
degree $2$ generator of $H_*(\C P^{\infty})$.
The two long exact sequences form a commutative diagram in which one square has the form
$$
\begin{CD}
SH^{+,S^1}(M) @>\ \ \ \cap e \ \ \ >> SH^{+,S^1}(M)\\
@VV\partial V @VV\partial V\\
H(M)[t] @>\ \ \ \cap e \ \ \ >> H(M)[t]
\end{CD}
$$
where on $H(M)[t]$, capping with the Euler class is division by $t$. 
This is equivalent to Theorem \ref{main}, if on the linearized contact homology side,
the map $\partial$ is given by gravitational descendants. More precisely, 
if $c \in HC(\partial M)$ and $b \in H_*(M)$ is the Poincar\'{e} dual of $\theta\in H^*(M,\partial M)$, 
then the coefficient of $b\otimes t^a$ in $\partial(c)$ is $\langle\tau_{a}(\theta)\rangle(c)$. 
However in this paper we do not prove this correspondence and instead work purely on the side of
linearized contact homology.}
\end{rmk}

\begin{rmk}\emph{Equation \eqref{maineq} was first observed in \cite{He} for  
subcritical Stein manifolds with vanishing first Chern class. The proof 
however relied on explicit knowledge of the structure of the moduli spaces. We
remove the vanishing Chern class condition in this paper. Furthermore, the
proof is sufficiently general so that the subcritical assumption is only needed
to avoid the technical issue of \emph{bad} orbits. Theorem \ref{main} is expected
to hold for a much wider class of exact symplectic fillings. We keep the subcritical
assumption for a concise exposition of the main argument.}
\end{rmk}

\begin{rmk}\emph{It is well known that for $M$ closed, 
the genus-$0$ gravitational descendants satisfy relations 
known as the string, dilaton, and divisor equations. 
A further relation, known as topological recursion, together with those 
three equations, reduces descendant invariants to the Gromov--Witten 
invariants \cite{kont}. Theorem \ref{main} also allows the reconstruction of 
descendants from correlators. However Equation \eqref{maineq} is a 
not a consequence of the 
generalizations of the known relations in the closed case. 
When $M$ is closed, to compute a one-point 
invariant with a high power of $\psi$, it is necessary to 
introduce extra marked points. Hence $\langle\tau_{a}(\theta)\rangle$ depends
on the multi-point Gromov--Witten invariants. On the other hand, 
Theorem \ref{main} allows the one-point invariants to be computed without introduction of extra marked points.}
\end{rmk} 

\begin{rmk}\emph{Throughout this paper we assume the polyfold theory of Hofer, Zehnder and Wysocki, 
\cite{polyfolds}, \cite{polyfolds2}, \cite{polyfolds3}, which 
forms the analytical foundation of SFT. Therefore we will treat all moduli spaces
as being transversely cut out, knowing an abstract perturbation scheme exists
under which all moduli spaces become branched manifolds with boundaries and corners of the 
expected dimension. Any argument used 
on the moduli spaces can be applied without change to the
perturbed moduli spaces.}
\end{rmk}

This paper is organized as follows:
in section $2$ we define the relevant objects, and then in section $3$ we prove Theorem \ref{main}.

\section{Preliminaries}

\subsection{Linearized Contact Homology}

We will give a quick sketch of linearized contact homology following 
\cite{surgeryformula}. 
Novikov ring coefficient were not used in \cite{surgeryformula} but the
necessary modification is well known and can be found in \cite{EliashbergGiventalHofer}.  
Details can also be found in section $3$ of \cite{Oancea}.

Let $(V^{2n-1},\xi)$ be a contact manifold with a contact $1$-form $\alpha$, i.e.,
$(d\alpha)^{n-1}\wedge \alpha$ is a volume form and $\xi = \mbox{Ker}(\alpha)$. The 
\emph{Reeb vector field} is the unique vector field $R$ such that
$$ d\alpha (R,-) = 0,\ \ \ \ \ \alpha(R)=1.$$
The flow of the Reeb vector field preserves the contact structure $\xi$. A (possibly multiply covered) 
Reeb orbit $\gamma$ is \emph{non-degenerate} if the linearized Poincar\'{e} return map of the Reeb flow 
has no eigenvalue equal to $1$. For a generic choice of $\alpha$, there are countably many closed Reeb
orbits, all of which are non-degenerate. Let $\kappa_{\gamma}$ denote the multiplicity of the orbit $\gamma$.

\begin{defn}
\emph{
A Reeb orbit is \emph{good} if it is not an even multiple of another orbit $\gamma$ such that the linearized Poincar\'{e}
return map along $\gamma$ has an odd total number of eigenvalues (counted with multiplicity) in the interval $(-1,0)$.}
\end{defn}

The \emph{symplectization} of a contact manifold $(V^{2n-1},\xi, \alpha)$ is 
the manifold $V\times \mathbb{R}$ with the 
symplectic form $d(e^t\alpha)$, where $t$ is the coordinate of $\mathbb{R}$. 

An almost complex structure $J$ on a symplectization 
$(V\times \mathbb{R}, d(e^t\alpha))$ is \emph{compatible} if 
\begin{itemize}
\item $J^2 = -\mbox{Id}$, 
\item $d\alpha(v, Jv)>0$ for all non-zero $v\in \xi$,
\item $J$ is invariant under translation in the $\mathbb{R}$-direction, 
\item $J\xi = \xi$, and $J\del{t} = R$.
\end{itemize}

A \emph{symplectic filling} $(M,\omega)$ of a contact manifold $(V,\xi,\alpha)$ is an open symplectic manifold 
with one open cylindrical end of the form $E = V\times [0,\infty)$.
On the cylindrical end, $\omega|_{E} = d(e^t\alpha)$. The complement of $E$ is compact. A 
filling is \emph{exact} if $\omega$ is exact. 

An almost complex structure $J$ on a filling $(M,\omega)$ is \emph{compatible} if
\begin{itemize} 
\item $J^2=-\mbox{Id}$, 
\item $\omega(v,Jv)>0$ for all non-zero $v\in TW$,
\item on the cylindrical end $E$, $J$ is invariant under translation in the $\mathbb{R}$-direction, 
\item on $V=V\times\{0\}$, $J\xi = \xi$, and $J\del{t} = R$.
\end{itemize}

We are only interested in the Reeb orbits on $\partial M$ which are contractible 
in $M$. For each such orbit $\gamma$, choose a capping disk $D_\gamma$ in $M$. There is a homotopically
unique symplectic trivialization of the contact distribution $\xi$ on $\gamma$, which, together
with the trivial symplectic subbundle spanned by the Reeb field $R$ and Liouville 
field $Y$, extends to a symplectic trivialization of $TM$ on $D_\gamma$.
With respect to this trivialization, the linearized Reeb flow along $\gamma$ defines a path of
symplectic matrices, which has a Conley--Zehnder index $\mu({\gamma})$.

A finite energy holomorphic curve in the symplectization $\partial M\times \R$ is a holomorphic
map $u$ from a punctured Riemann surface $(S,i)$ with some positive punctures 
$\{p_i^+\}$ and negative punctures $\{q_j^-\}$. 
Near each positive puncture, the map $u$ is asymptotic to a trivial cylinder $\gamma^+_i\times \R$
at the positive end of $\partial M\times \R$ for some Reeb orbit $\gamma^+_i$. Similarly negative punctures
are asymptotic to Reeb orbits at negative infinity. We then take equivalence 
classes of such holomorphic maps where $(u, (S,i))\cong (u', (S',j))$ if there is a diffeomorphism
$\phi\colon S\to S'$ such that $u=u'\circ \phi$ and $i=\phi^* j$.

Each holomorphic curve carries a 
homology class $A\in H_2(M)$, obtained by add to $u$ the capping disks 
$\{-D_{\gamma_i}\}$ for the positive punctures
and $\{D_{\gamma_j}\}$ for the negative punctures. 
Homology class for a holomorphic curve in the filling $M$ is defined identically, 
where only positive punctures exist.

\begin{rmk}\emph{For readability we omit some details
such as asymptotic markers and the combinatorial coefficients they bring, only
commenting on them when necessary. We will pretend our orbits are simple
unless otherwise stated.}
\end{rmk}

Let $\M^{A}(\Gamma^+,\Gamma^-)$ denote the moduli space of genus-$0$ holomorphic curves in $\partial M\times \R$,
in the homology class $A$, asymptotic to the set of orbits $\Gamma^+=\{\gamma^+_i\}_{i=1}^k$ at positive infinity, 
and $\Gamma^-=\{\gamma^-_i\}_{i=1}^l$ at negative
infinity, where $\Gamma^-$ is allow to be the empty set. 
The expected dimension of $\M^{A}(\Gamma^+,\Gamma^-)$ is given by the index formula
$$\dim \M^{A}(\Gamma^+,\Gamma^-) = \sum_{i=1}^k\mu(\gamma^+_i) - \sum_{i=1}^l\mu(\gamma^-_i) + 2\langle c_1(TM),A\rangle + (n-3)(2-k-l)$$
where $n$ is the complex dimension of $M$.

The index formula holds for holomorphic curves in the filling $M$ as well, 
where $\Gamma^-$ is always empty. Note that
holomorphic curves in $\partial M\times \R$ always come in one-parameter families since the almost complex structure
$J$ is invariant under translation in the $\R$ direction. Therefore we will mod out by this $\R$-action, when we refer to 
a rigid curves in the symplectization, the index is $1$.

The chain group for linearized contact homology of $\partial M$ is freely
generated by the good Reeb orbits over the the Novikov coefficient ring $\Lambda_{\omega}$. 
Define a grading on $\Lambda_{\omega}$ by $$|e^A|:= -2\langle c_1(TM), A\rangle.$$ 
And a grading on the orbits by 
$$|\gamma| = \mu(\gamma)+(n-3).$$

The differential is given by the count of rigid holomorphic cylinders in the symplectization
$\partial M \times \R$ \emph{anchored in $M$}. Let $\gamma^+$ and $\gamma^-$ be two Reeb orbits.
An element $u$ in the moduli space $\M^{A}_c(\gamma^+,\gamma^-)$ is a pair of holomorphic curves 
$u=(u_1,u_2)$ such that
\begin{itemize}
\item $u_1$ is a genus-$0$ holomorphic curve in the symplectization $\partial M\times \R$ with
one positive puncture asymptotic to $\gamma^+$, one distinguished negative puncture asymptotic to
$\gamma^-$, and any number of other negative punctures asymptotic to orbits $\gamma_1,\dots,\gamma_k$.
$u_1$ has homology class $B$.
\item $u_2$ is a collection of holomorphic planes $u_2^1,\dots u_2^k$ in the filling $M$ 
asymptotic to the Reeb orbits $\gamma_1,\dots,\gamma_k$. The total homology class is $C$.
\item $A=B+C$.
\end{itemize}

The expected dimension of $\M^{A}_c(\gamma^+,\gamma^-)$ is $\mu(\gamma^+)-\mu(\gamma^-)+2\langle c_1(TM),A\rangle$.
We are interested in rigid pairs, hence $u_1$ has index $1$ and $u_2$ is a collection of 
rigid planes. 

The linearized contact homology differential is given by
$$\partial \gamma = \sum_{(\gamma',A)\colon \mu(\gamma')-2\langle c_1(TM),A\rangle =\mu(\gamma)-1} n_{\gamma,\gamma'}e^{A}\gamma'$$
where $n_{\gamma,\gamma'}$ is the algebraic count of $\M^{A}_c(\gamma,\gamma')/\R$.

\begin{rmk}\emph{Asymptotic markers will introduce a factor of $\frac{1}{\kappa_{\gamma'}}$
into $n_{\gamma,\gamma'}$. This factor comes from the fact that when asymptotic makers
are include, there are $\kappa_{\gamma'}^{2}$ ways to glue along $\gamma'$ but only result in
$\kappa_{\gamma'}$ distinct glued curves. Please see \cite{EliashbergGiventalHofer} for
more detail.}
\end{rmk}

\subsection{Gravitational Descendants}
If we generalize the definition for Gromov--Witten invariants and their gravitational
descendants directly to the SFT setting, we immediately run into the 
problem that the moduli
spaces in SFT have codimension-$1$ boundary, therefore even though the tautological
line bundle $L$ still exists, its Chern class is not well defined. Instead the $\psi$ class should be viewed as the zero set of a generic section $s$ of $L$. The main idea is that these sections should not be chosen completely independently across different moduli spaces, but should satisfy a compatibility condition on 
the boundary. This is carried out in full detail in \cite{Fabert}. 

We will keep to simple case of one marked point. Let $\M^{A}(\gamma)$ be the moduli space
of holomorphic planes in $M$ asymptotic to the Reeb orbit $\gamma$ in homology class
$A$. Then
$\psi$ is the zero sets of a collection of generic coherent sections 
$\{s(\M^{A}(\gamma))\}$ over all orbits and homology classes, 
which is compatible with the restriction maps to boundary strata. For example, if a boundary stratum of $\M^{A}(\gamma)$ is of the form 
$\M^{A-B}(\gamma,\gamma')\times \M^{B}(\gamma')$, 
then $s(\M^{A}(\gamma))$ restricted to that stratum
is the pull back of $s(\M^{B}(\gamma'))$ under the projection map $\M^{A-B}(\gamma,\gamma')\times \M^{B}(\gamma') \to \M^{B}(\gamma')$. 
In general if there are more than one marked points, there 
will be further compatibility conditions for the symmetry of relabeling marked points.

Higher powers of $\psi$ are inductively defined. More precisely, $\psi^l$ is the zero
sets of a collection of coherent sections of $L^{\otimes l}$ over the zeros sets representing $\psi^{l-1}$, with a factor of $\frac{1}{l}$ 
since $c_1(L)=\frac{1}{l}c_1(L^{\otimes l})$.

For a compactly supported form $\theta \in H^*(M,\partial M)$,
the gravitational descendant $$\int_{\M^{A}(\gamma)}\ev^{*}(\theta)\cup \psi^l$$ is
defined to be the integral of $\ev^{*}(\theta)$ over the subset representing $\psi^l$.
Note that this subset is of codimension-$2l$ in $\M(\gamma)$.

The value of such a descendant will not be an invariant. However certain linear
combinations of descendants will be, if we make sure their codimension-$1$ boundary
strata cancel out.

\begin{prop}\label{tech}
Let $(M, \partial M)$ be an exact symplectic filling. 
If $a = [\sum_{i=1}^{k} c_ie^{A_i}\gamma_{i}]\in HC(\partial M)$ 
is a cycle in linearized contact homology, $A\in H_2(M)$ a homology class,
and $\theta$ a compactly supported closed form on $M$, 
then the value of the linear 
combination of descendants 
$$\langle\tau_l(\theta)\rangle^{A}(a)\colon =
\sum_{i=1}^{k}c_{i}\int_{\mathcal{M}^{-A-A_i}(\gamma_i)}\ev^*(\theta)\cup \psi^{l}$$
is independent of all choices. 
\end{prop}
\begin{proof}
It is enough to prove the case $l=0$. The higher $\psi$ classes will follow in
identical fashion, replacing each moduli space $\M^A(\gamma)$ by the codimension $2l$ 
zero sections representing $\psi^l$ in $\M^A(\gamma)$. By the compatibility condition
for coherent sections, the boundary strata of these zero sections have the same 
structures as their parent moduli spaces.

We will show that under the evaluation map, 
$\sum c_i \ \ev (\M^{-A-A_i}(\gamma_i))$ is a relative homology cycle in 
$H_*(M,\partial M)$, which we can then intersect with the Poincar\'{e} dual of
$\theta$. This intersection number is the same as the descendant integral, and 
is independent of choices.

By the compactness theorem of \cite{compactness}, a codimension-$1$ stratum of $\mathcal{M}^{-A-A_i}(\gamma_i)$ consists of $2$-story curves $(u_1, u_2)$, 
such that $u_1\in \mathcal{M}^{-A-A_i-B}(\{\gamma\}; \{\beta_1,\dots, \beta_l\})$ 
is a genus-$0$ holomorphic curve in the symplectization 
$\partial M\times \mathbb{R}$ with several negative punctures; 
and $u_2$ consists of $l$ holomorphic planes in the filling $M$, 
one asymptotic to each $\beta_i$ and in homology class $B_i$, 
with total homology class $B$. 
The marked point can be located on $u_1$ or on any one of the 
holomorphic planes of $u_2$. 

Consider the image of such a stratum under the evaluation map. 
If the marked point is on $u_1$, 
then the stratum is mapped to infinity 
(in other words to $\partial M$), which does not affect $H_*(M,\partial M)$.
Suppose the marked point lies on the
plane asymptotic to $\beta_1$. If the other planes in $u_2$ are not rigid, then
the dimension of $\M^{B_{1}}(\beta_1)$ is at least $2$ less than that of $\mathcal{M}^{-A-A_i}(\gamma_i)$.
Therefore this stratum is mapped under the evaluation map to a 
chain of  codimension at least $2$, so has no effect on homology.

The only codimension-$1$ strata must have $u_1$
an index-$1$ holomorphic curve between $\gamma$ and $\gamma'$, and $u_2$ contains
all but one rigid holomorphic planes. Evaluation map takes this stratum to 
$\ev(\M^{B_1}(\beta_1))$. This is exactly the configuration appearing
in the linearized contact homology differential. Since $a$ is a cycle, there
is another boundary stratum in some $\mathcal{M}^{-A-A_j}(\gamma_j)$ which 
evaluates to $-\ev(\M^{B_1}(\beta_1))$. Therefore 
$\sum c_i \ \ev (\M^{-A-A_i}(\gamma_i))$ is a relative homology cycle.
\end{proof}

\begin{rmk}\emph{The unfortunate negative sign for the homology 
class $A$ is due to the fact that 
$|e^A\gamma| = \mu(\gamma)-2\langle c_1(TM), A\rangle+(n-3)$, 
but the moduli space $\M^{A}(\gamma)$ has dimension 
$\mu(\gamma)+ 2\langle c_1(TM), A\rangle +(n-3)$.}
\end{rmk}

\begin{defn}\emph{Define
$\langle\tau_{l}(\theta)\rangle \colon HC(\partial M) \rightarrow \Lambda_{\omega}$ 
by 
$$\langle\tau_{l}(\theta)\rangle (a) = 
\sum_{A\in H_2(M)} \langle\tau_{l}(\theta)\rangle^A(a)\cdot e^A$$ }
\end{defn}

\subsection{Stein Domains}
An open complex manifold $(M^{2n},J)$ is \emph{Stein} if
it can be realized as a properly embedded complex submanifold
of some $\mathbb{C}^{N}$. 
A smooth function $f\colon M\rightarrow \mathbb{R}$ is \emph{exhausting}
if it is proper and bounded from below. Let $d^{J}f$ denote
$df\circ J$. The function $f$ is
\emph{plurisubharmonic} if the associated 2-form $\omega_{f} = -dd^{J}f$
is a symplectic form taming $J$, i.e., $\omega_{f}(v,Jv)>0$ for every
non-zero tangent vector $v$. Plurisubharmonicity is an open condition.
We can therefore assume $f$ to be Morse. By a theorem of Grauert, an
open complex manifold is Stein if and only if it admits a plurisubharmonic
function.

A Stein manifold $(M^{2n},J)$ with an exhausting plurisubharmonic
function $f$ admits the following associated structures:

\begin{itemize}
\item a symplectic form $\omega_{f} = -dd^{J}f$ which is $J$-invariant,
\item a primitive $\alpha = -d^{J}f$,
\item a vector field $Y$ such that $\alpha=\iota_{Y}\omega$,
\item a metric $g(v,w)=\omega(v,Jw)$.
\end{itemize}

Since $L_{Y}\omega = \iota_{Y}d\omega + d(\iota_{Y}\omega) = d\alpha =\omega$,
the vector field $Y$ is Liouville, i.e., the flow of $Y$ expands the symplectic
form. In fact $Y$ is the gradient vector field of $f$ with respect to the metric $g$.
The unstable submanifolds of the critical points are isotropic, therefore the
Morse index is at most $n$.

A \emph{Stein domain} $(M,\partial M)$ is a compact submanifold of a Stein manifold $W$ of the 
same dimension, such that $\partial M$ is transverse to the Liouville
vector field $Y$. A Stein domain becomes a symplectic filling if we complete
it by attaching the symplectization $\partial M\times [0,\infty)$ to the outside of
$(M,\partial M)$. Note that such a Stein filling is always exact. 
We will abuse notation and denote the completion by $(M,\partial M)$ also. 

\begin{defn}
\emph{A Stein domain $M^{2n}$ is \emph{subcritical} if there exists a plurisubharmonic Morse function with all critical points having Morse index strictly less than $n$;
and is of \emph{finite type} if the number of critical points is finite.}
\end{defn}

\begin{rmk}
\emph{For a subcritical Stein manifold of finite type, 
we may assume that all orbits are good by action filtration. More precisely,
we can construct a decreasing sequence of contact $1$-forms $\alpha_n$ with an increasing sequence of real 
numbers $b_n\to \infty$ such that all Reeb orbits for $\alpha_n$ with action less than $b_n$ are good. Then the linearized
contact homology is the limit of $HC^{\leq b_n}(\partial M, \alpha_n)$. From now on we will assume that all orbits are good.}
\end{rmk}

The core of a Stein domain $M$ with a plurisubharmonic function $f$, $C_M$, 
is the union of all the unstable submanifolds of $f$ in $M$. 
The integrable complex structure on $M$ is not compatible in the sense of a symplectic
filling, because it is generally not invariant under the flow of the Liouville field.
However perturbations can be chosen to be made only on the cylindrical end 
of the filling, keeping $J$ integrable near the core of the manifold. Since $C_M$
consists of cells of dimension at most $n$, but $M$ has dimension $2n$, standard 
homotopy theory shows there is no obstruction in choosing a trivial complex line
subbundle of $TM$ over $C_M$. Hence there is a consistent way to pick a
complex line $\C_p$ at each $p\in C_M$.

In this setting there is a geometric interpretation of the gravitational 
descendants. For each $\theta\in H^*(M,\partial M)$,
choose a homology cycle representing the Poincar\'{e} dual
of $\theta$, say $\alpha$, such that $\alpha$ lies in $C_M$. 
First we restrict to curves passing through $\alpha$. Let
$\M^{A}(\gamma;\alpha)$ be the moduli space of holomorphic planes with
one marked point, in the homology class $A$, asymptotic to $\gamma$,
and the marked point is mapped to $\alpha$. We can normalize so
that the marked point is $0$.

For $u\in \M^{A}(\gamma;\alpha)$, choose a small neighborhood of 
$u(0)$ where $J$ is integrable,
then in that neighborhood project $u$ onto the chosen complex line $\C_{u(0)}$. 
Denote this projection by $\rho_{u(0)}$.
Then $\rho_{u(0)}\circ u$ is a holomorphic map $\C\to \C$ and we can compute 
its vanishing order at $0$. If the first $l$ derivative vanish then we say $u$ has 
ramification index $(l+1)$. A coherent collection of sections for gravitational
descendants can be chosen to be the curves of ramification index $l$.

To be more precise, the domain of our holomorphic curves, 
$\C$ with one marked point, is not stable. The automorphism group is exactly $\C^*$. 
The addition of an extra marked 
point makes the domain stable. We will call the space of holomorphic planes with one marked 
point the \emph{unparametrized} curves, and the space of holomorphic planes with two marked points
the \emph{parametrized} curves.

The tautological line bundle $L$ over $\M^{A}(\gamma;\alpha)$ is the dual of the bundle of 
parametrized curves over unparametrized curves. The fiber over each element $u$ of
$\M^{A}(\gamma;\alpha)$ consists of a holomorphic map $\tilde{u}\colon \C \rightarrow M$, together with the
$\C^*$-family of reparametrizations of $\tilde{u}$, $\{\tilde{u}(cz),c\in \C^*\}$. We can fill in
the zero section by observing that $c=0$ corresponds precisely
to the nodal curve $\C\cup \C P^1\rightarrow M$, where a constant ghost bubble is attached to $u$.

The derivative at $0$ of the projection of $u$ to the chosen
complex direction $\mathbb{C}_{u(0)}$, 
$\tilde{u} \rightarrow \frac{\partial}{\partial z}(\rho_{u(0)}\circ \tilde{u})|_{z=0}$, 
is a section of the dual
of the tautological line bundle. Furthermore such sections form a coherent collection over different $\M^{A}(\gamma;\alpha)$'s. The
zero set of this section consists of (unparametrized) holomorphic curves $u$ whose representative $\tilde{u}$,
after projection onto the chose $\C$ direction, has the form $z\rightarrow cz^{k}$ for some $k\geq 2$.
We will often suppress the chosen complex direction, and simply refer to this zero set as curves with 
ramification index $2$.

Similarly, over the curves with ramification index $2$, 
$\tilde{u}\rightarrow \frac{\partial^2}{\partial z^2}(\rho_{u(0)}\circ \tilde{u})|_{z=0}$ is a section 
of $L^{\otimes 2}$. The zero set of this section is referred to as curves with ramification index $3$.

Therefore if $\theta$ is Poincar\'{e} dual to the point class, then 
$\int_{\mathcal{M}_{\gamma}}\ev^*(\theta)\wedge \psi^{l}$ can be interpreted as the count of holomorphic
planes passing through $p$ with ramification index $(l+1)$, divided by $l!$.

\begin{rmk}\emph{This part of the setup works without any subcritical assumption.}
\end{rmk}

\section{Proof of Main Theorem}

We first give the description of the map $D$ following section 7.2 of \cite{Oancea}. 
The map $D$ is described exclusively in terms of holomorphic 
curves in the symplectization $\partial M\times \R$ anchored in $M$. We
assume that there are no bad orbits. In the presence bad orbits this 
description will need to be modified to include their contribution.
Note that this is the only place where we use the subcritical assumption.

Identify the geometric image of each Reeb orbit $\gamma$ with the unit circle $S^1\cong \R/2\pi \Z$ in
$\C$, i.e. choose a point $P_\gamma$ which is identified with $0$, and the 
Reeb flow is counterclockwise on $S^1$. Denote this by $S^1_{\gamma}$.

Let $u$ be an element of $\M^A_c(\gamma^+,\gamma^-)$. Take
a \emph{parametrized} anchored holomorphic cylinder $(u_1,u_2)$ 
between $\gamma^+$ and $\gamma^-$ representing $u$. 
Choose a global polar coordinate $\C\setminus \{0\}$ for the domain of $u_1$, where the
positive puncture is $\infty$, the distinguished negative puncture is $0$, and the
other negative punctures $\{z_i\}_{i=1}^{k}$ are on $\C\setminus \{0\}$. 
Let $H$ be the positive real axis.

Since $u_1$ is a curve in symplectization, it
has the form
$$u_1 = (\overline{u}_1, f)\colon \C\setminus \{0, z_1,\dots, z_k\}\rightarrow \partial M\times \R.$$

Reparametrize the domain cylinder $\C\setminus \{0\}$ of $u_1$ by rotation, such that 
$$\lim_{z\rightarrow \infty, z\in H}\overline{u}_1(z) = P_\gamma^+.$$ 
Then define 
$$\ev^-(u) = \lim_{z\rightarrow 0, z\in H}\overline{u}_1(z) \in S^1_{\gamma^-}.$$

Similarly after making 
$\lim_{z\rightarrow 0, z\in H}\overline{u}_1(z) = P_\gamma^-$, we define 
$$\ev^+(u) = \lim_{z\rightarrow \infty, z\in H}\overline{u}_1(z) \in S^1_{\gamma^+}.$$

\begin{rmk}\label{asymp}\emph{
Since $S^1_{\gamma}$ is only the geometric image of $\gamma$, it appears that when $\gamma^+$ is
multiply covered, the reparametrization to make $\lim_{z\rightarrow \infty, z\in H}\overline{u}_1(z) = x$
is not unique, so $\ev^{\pm}(u)$ is not well defined. This is due to us ignoring the asymptotic markers.
The full definition of $\M_c^A(\gamma,\gamma')$ includes asymptotic markers at each puncture.
Equivalence of holomorphic maps has to also preserve the makers. In particular this means
each element in our moduli space without asymptotic markers actually corresponds to 
$\kappa_{\gamma^+}\kappa_{\gamma^-}$ elements in the moduli space with markers. Every possible way to
reparametrize is covered by exactly one such curve with markers. We will continue pretending the
orbits are simple.} 
\end{rmk}

The map $D$ is induced by the chain level map
\begin{equation}\label{eq:def1}
\Delta(\gamma)=\frac{1}{\kappa_{\gamma'}}\sum_{|\gamma|=|e^{A}\gamma|+2} c^A_{\gamma,\gamma'}e^A \gamma', 
\end{equation}
where $c^A_{\gamma,\gamma'}$ is the sum of counts of two types of moduli spaces:
\begin{enumerate}
\item the moduli space $\M_{1}$ of anchored holomorphic cylinders $u$ in $\M^{A}_c(\gamma, \gamma')$
such that $\ev^+(u) = 0$, or equivalently $\ev^-(u) = 0$.
\item the moduli space $\M_2$ of parametrized broken holomorphic cylinders 
$$u=(u_1,u_2)\in \M^B_c(\gamma, \beta), \ \ u'=(u'_1,u'_2)\in \M^{A-B}_c(\beta, \gamma')$$
such that on the intermediate breaking orbit $S^1_{\beta}$,
$\{0, \ev^-(u),  \ev^+(u')\}$ lie in anticlockwise order.
\end{enumerate}

The moduli space $\M_c^A(\gamma,\gamma')$ consists of one parameter family of holomorphic cylinders,
the boundary consists of a number of broken cylinders. Intuitively $\M_1$ counts the solutions of $\ev^-(u)=0$ 
in the interior of $\M_c^A(\gamma,\gamma')$, and $\M_2$ counts a subset of $\partial \M_c^A(\gamma,\gamma')$.

\begin{rmk}\emph{At first glance it seems $\M_2$ depends how we identify $\gamma$ with $S^1$, i.e., the
the choice of $P_\gamma$ to set as $0$. It is true that chain map $\Delta$ depends on such choices. 
However $D$ is invariantly define on homology. This can be easily seen as follows, suppose we 
let the $P_\gamma$'s move continuously. Then three different types of change can occur 
on some intermediate breaking orbit $\beta$:
\begin{itemize}
\item for some $u_0\in\M_c^{B}(\gamma,\beta)$, $\ev^-(u_0)$ moves pass $0$ in clockwise 
direction. Then for all other $u'\in\M_c^{C}(\beta,\gamma')$, where $\gamma$ and $C$ 
vary over all possibilities, the broken cylinders $\{u_0,u'\}$ go from counting in $\M_2$ to not counting. 
Nothing else is changed. Hence $\Delta(\gamma)$ is changed by a boundary term $\del \beta$, and therefore
has no effect on homology. 
\item for some $u_0\in\in\M_c^{B}(\gamma,\beta)$ and $u'_0\in\M_c^{A-B}(\beta,\gamma')$, $\ev^-(u_0)$
moves pass $\ev^+(u'_0)$ in clockwise direction. Then the only change is $\{u_0,u'_0\}$ goes from
not counting to counting. However, in the process, we crossed a cylinder where $\ev^-(u_0)=\ev^+(u'_0)$.
After gluing, the cylinder $\tilde{u} = u_0\# u'_0 \in \M_c^{A}(\gamma,\gamma')$ satisfies
$\ev^-(\tilde{u})=0$. Hence $\tilde{u}$ belongs to the moduli space $\M_{1}$. 
Therefore there is a corresponding change in $\M_{1}$ which 
cancels with this change in $\M_{2}$, so overall $\Delta(\gamma)$ remains unchanged.
\item for some $u'_0\in\M_c^{C}(\beta,\gamma')$, $0$ moves pass $\ev^+(u'_0)$ in
clockwise direction. Then similar to the first case, for every $u\in\M_c^{B}(\gamma,\beta)$, 
the broken cylinders $\{u,u'_0\}$ goes from counting to not counting. If 
$a=\sum_{i=1}^{k} c_i\gamma_i, c_i\in \Lambda_{\omega}$ is a cycle in linearized 
contact homology, then by definition, each cylinder $u_0\in \M_c^{B_i}(\gamma_i,\beta)$ 
will be cancel by some other cylinder $u_1\in \M_c^{B_j}(\gamma_j,\beta)$. It follows that 
$\Delta$, when evaluated on a cycle, remains unchanged.
\end{itemize}}
\end{rmk}

\begin{proof}[\emph{\textbf{Proof of Theorem \ref{main}}}]
As in the proof of Proposition \ref{tech}, it is enough to prove equation \eqref{maineq}
for the case $l = 1$. For higher values of $l$ we simply repeat the argument with
the coherent codimension-$2l$ zero sections of $L^{\otimes l}$ in place of the moduli spaces.
As in Section $2.3$, we can chose $l$-th derivative to be the coherent sections, and
the curves with ramification index-$l$ are the coherent codimension-$2l$ zero sections sections. 

Consider the gravitational descendant $\langle\tau_1\theta\rangle^{A}$. Let 
$a=\sum_{i=1}^{k} e^{A_i}\gamma_i$ be a cycle in linearized contact 
homology. Let $\alpha$ be a homology cycle representing the Poincar\'{e} dual of $\theta$. 
For each $\langle\tau_1\theta\rangle^{A}(e^{A_i}\gamma_i)$, 
we are interested in the number of holomorphic 
planes asymptotic to $\gamma_i$, 
in the homology class $-A-A_i$,  
and passing through $\alpha$ with ramification index $2$. 

Reduce the ramification index requirement by one, then the moduli space becomes
two dimensional. In this case we have $\dim \M^{-A-A_i}(\gamma_i; \alpha) = 2$.
Recall that $\M^{X}(\gamma; \alpha)$ is the moduli space of 
holomorphic planes with one marked point in homology class $X$, asymptotic to $\gamma_i$, 
and the marked point is mapped to $\alpha$.

First we trivialize the tautological bundle $L$ over $\M^{-A-A_i}(\gamma_i; \alpha)$.
As explained in Section $2.3$, this amounts to choosing a parametrized map 
$u\colon \C\rightarrow M=M'\times \C$ over each element of $\M^{-A-A_i}(\gamma_i; \alpha)$
(more precisely this gives a trivialization of the dual of $L$ and hence $L$). 

To fix the $S^1$-component of the automorphism group, 
similar to the previous discussion on anchored cylinders,
we may required that $\lim_{z\rightarrow \infty, z\in H}u(z) = P_{\gamma}$ (note that this 
does not uniquely determine $u$ if $\gamma$ is multiply covered, but Remark \ref{asymp} applies
here as well). To fix the $\R$-component of the automorphism group, we can simply require 
that $u(1)$ is always a distance of $\epsilon$ away from $u(0)$, for some sufficiently small 
constant $\epsilon$.

Under this trivialization of $L$, the coherent section is a map from 
$\M^{-A-A_i}(\gamma_i; \alpha)$ to $\C$:
$$f\colon u \longmapsto \frac{\partial}{\partial z}\left(\rho_{u(0)}\circ u\right)(0)$$

Then $\langle\tau_1\theta\rangle^{A}(e^{A_i}\gamma_i)$ is the number of zeroes of
$f$. Since $\M^{-A-A_i}(\gamma_i; \alpha)$ is $2$-dimensional, its boundary consists of a 
collection of circles, and the number of zeroes of $f$ is equal to the winding number
of $f({\partial \M^{-A-A_i}(\gamma_i; \alpha)})$ around $0$.

The boundary of $\M^{-A-A_i}(\gamma_i; \alpha)$ consists of $2$-story curves
$(u_1, u_2)$ where 
\begin{itemize}
\item $u_1\in \M^{B}(\{\gamma_i\},\{\beta_1,\dots \beta_k\})$ 
is a genus-$0$ holomorphic in the symplectization $\partial M\times \R$,
with one positive puncture asymptotic to $\gamma_i$, and several negative punctures
asymptotic to $\{\beta_j\}$.
\item $u_2$ is a collection of holomorphic planes $u_2^j\in \M^{C_j}(\beta_j)$ 
in $\M$ asymptotic to the negative punctures of $u_1$. One of the planes, let us always
use $u_2^1$, contains the marked point. Thus$u_2^1$ belongs to the cut down moduli space 
$\M^{C_1}(\beta_1;\alpha)$, other planes are unconstrained.
\item $B+C_1+\dots +C_k = -A-A_i$.
\end{itemize}

The total index of $(u_1,u_2)$ is $2$. There are a few different ways the
index can distribute on the two stories:
\begin{enumerate}
\item
$(u_1,u_2)$ has index-$(1,1)$: $u_1$ is rigid up to $\R$-translation.
And the plane $u_2^1$ comes in a $1$-parameter family, the others plane 
$u_2^j$'s are rigid. 
\item
$(u_1,u_2)$ has index-$(1,1)$: $u_1$ is rigid. However $u_2^1$ is also rigid, some plane
$u_2^j$ without marked point has index-$1$.
\item
$(u_1,u_2)$ has index-$(2,0)$: $u_1$ is a $1$-parameter family. And all the $u_2^i$'s
are rigid. 
\end{enumerate}

Therefore the boundary circles of $\M^{-A-A_i}(\gamma_i; \alpha)$ is divided into arcs and 
circles of these three types. Observe that the end points of these arcs,
i.e. to cross from one type to another, are exactly the $3$-story curves $(u_1, u_2, u_3)$,
such that all three are rigid curves: 
\begin{itemize}
\item
$u_1$ is a curve with one positive puncture and
several negative punctures in the symplectization $\partial M\times \R$; 
\item
$u_2$ is also in the symplectization, it has several connected components, 
one of which is a curve with one positive puncture and
several negative punctures, all other components are trivial cylinders;
\item
$u_3$ is a collection of rigid planes in the filling, with the marked point
on one of them.
\end{itemize}

If we glue $u_1$ and $u_2$ first and leave $u_3$ fixed, we get a curve of type (3). If 
we glue $u_2$ and $u_3$ and leave $u_1$ fixed, we get a curve of type (1) or (2).
Note that we cannot go from type (1) arcs directly to type ($2$) arcs.

Recall that the winding number for an arc $s\colon [0,1] \to S^1$ is equal to 
$\tilde{s}(1)-\tilde{s}(0)$ where $\tilde{s}$ is the lift of $s$ to the 
universal cover $\R$. Hence the value of 
$\langle\tau_1\theta\rangle^{A}(e^{A_i}\gamma_i)$ is the total winding number
of the derivative map $f$ on all three types of arcs.

For type (1) arcs or circles, observe that if we remove the plane $u_2^1$, then $(u_1,u_2)$ is exactly the
information for an anchored holomorphic cylinder between $\gamma_i$ and $\beta_1$, i.e.
$$(u_1,u_2)\in \M_c^{-A-A_i-B_1}(\gamma_i,\beta_1)\times \M^{B_1}(\beta;\alpha).$$ 
The winding of $f$ on this arc is determined by the $1$-parameter family of planes in $\M^{B_1}(\beta;\alpha)$.
However, we are interested not in $\langle\tau_1\theta\rangle^{A}(e^{A_i}\gamma_i)$, 
but $\langle\tau_1\theta\rangle^{A}(\sum_{i={1}}^{k}c_ie^{A_i}\gamma_i)$, which is a cycle in linearized
contact homology. Note that $\M_c^{-A-A_i-B_1}(\gamma_i,\beta_1)$ is exactly a term in the differential.
Therefore there will be another arc on the boundary of some other $\M^{-A-A_j}(\gamma_j; \alpha)$ which has
the form 
$$(u_1',u_2')\in \M_c^{-A-A_j-B_1}(\gamma_j,\beta_1)\times \M^{B_1}(\beta;\alpha),$$ 
for some $j$, but with the opposite sign.
It follows that the winding of $f$ on the type (1) arcs cancel out for $a\in HC(\partial M)$.

For type (2) arcs or circles, observe that $u_1$ is rigid, and so is $u_2^1$, therefore the gluing 
reparametrization is the same over the entire $1$-parameter family. Hence $f$ is constant and 
there is no contribution to the winding number.

For type (3) arcs or circles, once again after we remove the plane $u_2^1$, then $(u_1,u_2)$ is now
a $1$-parameter family of anchored holomorphic cylinder between $\gamma_i$ and $\beta_1$, exactly
as considered in the geometric description of the map $D$. When we glue $u_2^1$ to $u_1$, we must reparametrize
the domain of the rigid plane $u_2^1$ so that the limit of the 
positive $x$-axis of $u_2^1$ matches with $\ev^-(u_1)$. It follows that the winding number of $f$ on the
type (3) arcs or circles exactly agrees with the winding of number
$\ev^-(\M_c^{-A-A_i-B_1}(\gamma_i, \beta))$. Similar to the type (1) curves, we need to sum over all $i$.

Recall that the in the geometric description of the map $D$, moduli space $\M_1$ count the number of time
the arc $\ev^-(\M_c^{-A-A_i-B_1}(\gamma_i, \beta))$ crosses $0$. It can be thought of as an integral
approximation of the real winding number. We need to show that moduli space $\M_2$
exactly gives the right correction.

If a moduli space $\M_c^{-A-A_i-B_1}(\gamma_i, \beta)$ is a circle then the winding number is integral
already, and $\M_1$ counts correctly. If $\M_c^{-A-A_i-B_1}(\gamma_i, \beta)$ is a path, then as 
discuss above, the two end point are $3$-story curves. If an end is between a type (3) and a type (2) arc,
then since type (2) arcs do not change the winding number, we may continue onto the other end of
the type (2) arc, which must be a type (3) arc again. In other word, we may include type (2) 
curves as part of the type (3) curves. Therefore all ends of type (3) arcs are now between type (3)
and type (1) arcs. In particular, this means in the $3$-story curve $(u_1,u_2,u_3)$, the rigid 
plane $u_3^1$ with the marked point is not attached to a trivial cylinder in $u_2$. Hence both
$u_1$ and $u_2$ are anchored holomorphic cylinders. 

From now on we will use the short hand that $u(\delta_1, \delta_2)$ to be an anchored holomorphic
cylinder between the orbits $\delta_1$, and $\delta_2$ in the suitable homology class, and $u(\delta)$
to be a rigid plane asymptotic to $\delta$ with the constrain of passing through $\alpha$.

Suppose$\M_c^{-A-A_i-B_1}(\gamma_i, \beta)$ has two ends $(u_1(\gamma_i,\delta),u_2(\delta, \beta))$,
and $(u'_1(\gamma_i,\delta),u'_2(\delta,\beta))$. Furthermore assume 
$\M_c^{-A-A_i-B_1}(\gamma_i, \beta)$ is
oriented in that direction. Since $\M_1$ already counted the integral part of the winding number of 
$\ev^-(\M_c^{-A-A_i-B_1}(\gamma_i, \beta))$, the fractional part left is simply 
$$\ev^-(u'_1(\gamma_i,\delta),u'_2(\delta,\beta)) - \ev^-(u_1(\gamma_i,\delta),u_2(\delta, \beta))$$
where each number is normalized to be in the range $[0,2\pi)$.

We also have 
\begin{equation}\label{mod}
\ev^-(u_1(\gamma_i,\delta),u_2(\delta,\beta)) = \ev^-(u_1(\gamma_i,\delta)) + \ev^-(u_2(\delta,\beta))\mod 2\pi
\end{equation}
where $\ev^-(u_1(\gamma_i,\delta))$ and $\ev^-(u_2(\delta,\beta))$ are normalized to lie in $[0,2\pi)$ as well.
Similarly for $u'$. Let us assume for now the all $\ev^-(u(\gamma_1,\gamma_2))$ are small positive numbers, so that
\eqref{mod} is actual equality. 
This means every $\{0, \ev^-(u_1),  \ev^+(u_2)\}$ lies in clockwise order, so $\M_2$ is empty. 

On the other hand, in this case the sum of the fraction winding numbers is
\begin{equation}\label{sum}
\sum_{\gamma_i,\delta, \beta}\ev^-(u'_1(\gamma_i,\delta)) + \ev^-(u'_2(\delta,\beta))- \ev^-(u_1(\gamma_i,\delta)) - \ev^-(u_2(\delta,\beta))
\end{equation}

For each particular cylinder $v(\delta,\beta)$, since $a$ is a cycle in $HC(\partial M)$, every time $v(\delta,\beta)$
appears as part of the start of a moduli space $(u_1(\gamma_i,\delta),v(\delta, \beta))$, there is a canceling 
$(u'_1(\gamma_j,\delta),v(\delta, \beta))$ where it appears with the opposite sign as the end of a moduli space.
Hence the contribution of $\ev^-(v(\delta,\beta))$ is $0$ in the sum \eqref{sum}.

Similarly suppose a cylinder $v(\gamma_i,\delta)$ appears in $(v(\gamma_i,\delta),u_2(\delta, \beta))$ as the start of
a moduli space. Recall that there is a rigid plane $u(\beta)$ on the base of the $3$-story curve 
$$(v(\gamma_i,\delta),u_2(\delta, \beta),u(\beta)).$$ 
Glue $u_2(\delta, \beta)$ and $u(\beta)$, this is an end of a $1$-parameter family of planes 
asymptotic to $\delta$. Let the other end of this family be $(u'_2(\delta, \beta'),u'(\beta))$. 
Then 
$(v(\gamma_i,\delta),u'_2(\delta, \beta'),u'(\beta'))$ appears with the opposite sign to $(v(\gamma_i,\delta),u_2(\delta, \beta),u(\beta))$.
Therefore the contribution of each $\ev^-(v(\gamma_i,\delta))$ in the sum \eqref{sum} is
$0$ as well.

It follows that the fractional winding numbers is also $0$, if \eqref{mod} is actual equality.
Now it is easy to see that every time 
$\{0, \ev^-(u_1),  \ev^+(u_2)\}$ lies in anticlockwise order, it means 
$\ev^-(u_1(\gamma_i,\delta)) + \ev^-(u_2(\delta,\beta)) > 2\pi$, so in the sum \eqref{sum}
we replace $\ev^-(u_1(\gamma_i,\delta)) + \ev^-(u_2(\delta,\beta))$ by $\ev^-(u_1(\gamma_i,\delta)) + \ev^-(u_2(\delta,\beta)) -2\pi$
and therefore the winding number changes by one, exactly corresponding to the inclusion of that
broken cylinder to $\M_2$. 

We showed that geometric count of the moduli spaces $\M_1$ and $\M_2$ exactly matches the winding number of the
derivative map $f$. Theorem \eqref{main} follows immediately.\end{proof}

\newpage
\section*{Acknowledgements}
The author thanks F. Bourgeois, Y. Eliashberg K. Honda, S. Lisi and O. Van Koert 
for many helpful suggestions and valuable comments. This work was written
during the author's stay at the Simons Institute at Stony Brook University.
This work was partially supported by the ERC 
Starting Grant of Frederic Bourgeois StG-239781-ContactMath.


\newpage
\begin{thebibliography}{}

\bibitem[Bou04]{Bourgeois}
F. Bourgeois.
\emph{A Morse-Bott approach to contact homology}.
Ph.D. Thesis.

\bibitem[BEE12]{surgeryformula}
F. Bourgeois, T. Ekholm and Y. Eliashberg.
\emph{Effect of Legendrian surgery}.
Geom. \& Topol. {\bf 16} (2012), 301--389.

\bibitem[BEHWZ03]{compactness}
F.Bourgeois, Y. Eliashberg, H. Hofer, K. Wysocki and E. Zehnder.
\emph{Compactness results in symplectic field theory}.
Geom. \& Topol. {\bf 7} (2003), 799--888.

\bibitem[BO09]{Oancea}
F. Bourgeois and A. Oancea.
\emph{An exact sequence for contact and symplectic homology}.
Invent. Math. {\bf 175} (2009), no. 3, 611--680

\bibitem[BO]{S1symplectic}
\emph{The Gysin exact sequence for $S^1$-equivariant symplectic homology}.
Preprint 2009.

\bibitem[Cie02]{Cieliebak}
K. Cieliebak.
\emph{Subcritical manifolds are split}.
Preprint 2002.

\bibitem[Eli90]{EliashbergStein}
Y. Eliashberg.
\emph{Topological characterization of Stein manifolds of dimension $> 2$}.
Int. J. of Math. {\bf 1} (1990), no. 1, 29--46.

\bibitem[EGH00]{EliashbergGiventalHofer}
Y. Eliashberg, A. Givental and H. Hofer.
\emph{Introduction to symplectic field theory}.
Geom. Funct. Anal. 2000, 560--673.

\bibitem[Fab10]{Fabert}
O. Fabert.
\emph{Gravitational descendants in symplectic field theory}.
Comm. Math. Phys. {\bf 302} (2011), no. 1, 113--159. 

\bibitem[He12]{He}
J. He.
\emph{Correlators and descendants of subcritical Stein manifolds}.
Accepted.

\bibitem[HWZ96]{HoferWysockiZehnder}
H. Hofer, K. Wysocki and E. Zehnder.
\emph{Properties of pseudo-holomorphic curves in
symplectization. I. Asymptotics}.
Ann. Inst. H. Poincar\'{e} Anal. Non Linaire {\bf 13}
(1996), no. 3, 337--379.

\bibitem[HWZ06]{polyfolds}
H. Hofer, K. Wysocki and E. Zehnder.
\emph{A General Fredholm Theory I: A Splicing-Based Differential Geometry}.
To appear Journal of the European Mathematical Society.

\bibitem[HWZ07]{polyfolds2}
H. Hofer, K. Wysocki and E. Zehnder.
\emph{A General Fredholm Theory II: Implicit Function Theorems}.
Preprint 2007.

\bibitem[HWZ08]{polyfolds3}
H. Hofer, K. Wysocki and E. Zehnder.
\emph{A General Fredholm Theory III: Fredholm Functors and Polyfolds}.
Preprint 2008.

\bibitem[HWZ11]{GWpolyfold}
H. Hofer, K. Wysocki and E. Zehnder.
\emph{Applications of Polyfold Theory I: The Polyfolds of Gromov--Witten Theory}.
Preprint 2011.

\bibitem[KM98]{kont}
M. Kontsevich and Y. Manin.
\emph{Relations Between the Correlators of the Topological 
Sigma-Model Coupled to Gravity}.
Comm. Math. Phys. {\bf 196} (1998), no. 2, 385--398.

\end{thebibliography}
\end{document}